\documentclass[a4paper,12pt]{amsart}
\usepackage{amssymb}

\newtheorem{thm}{Theorem}
\newtheorem{cor}{Corollary}

\newtheorem{ex}{Example}

\newcommand{\arctg}{\operatorname{arctg}}
\newcommand{\real}{\operatorname{Re}}
\newcommand{\imag}{\operatorname{Im}}

\newcommand{\A}{{\mathcal A}}

\newcommand{\D}{{\mathbb D}}

\begin{document}
\bibliographystyle{amsplain}

\title[New criteria for starlikeness in the unit disc]{New criteria for starlikeness in the unit disc}

\author[M. Obradovi\'{c}]{Milutin Obradovi\'{c}}
\address{Department of Mathematics,
Faculty of Civil Engineering, University of Belgrade,
Bulevar Kralja Aleksandra 73, 11000, Belgrade, Serbia}
\email{obrad@grf.bg.ac.rs}

\author[N. Tuneski]{Nikola Tuneski}
\address{Department of Mathematics and Informatics, Faculty of Mechanical Engineering, Ss. Cyril and
Methodius
University in Skopje, Karpo\v{s} II b.b., 1000 Skopje, Republic of North Macedonia.}
\email{nikola.tuneski@mf.edu.mk}

%\date{\today  %June. 30, 09
%;  File: }

\begin{abstract}
It is well-known that the condition $\real \left[1+\frac{zf''(z)}{f'(z)}\right]>0$, $z\in\D$, implies that $f$ is starlike function (i.e. convexity implies starlikeness). If the previous condition is not satisfied for every $z\in \D$, then it is possible to get new criteria for starlikeness by using
$\left|\arg\left[\alpha+\frac{zf''(z)}{f'(z)}\right]\right|$, $z\in\D$, where $\alpha>1.$
\end{abstract}

\maketitle

\medskip

\section{Introduction and definitions}

\medskip

Let $\mathcal{A}$ be the class of functions $f$ which are analytic  in the open unit disc $\D=\{z:|z|<1\}$
of the form $f(z)=z+a_2z^2+a_3z^3+\cdots$,
and let $\mathcal{S}$ be the subclass of $\mathcal{A}$ consisting of functions that are univalent in $\D$.

\medskip

Also, let
\[
\begin{split}
\mathcal{C}&=\left[f\in\mathcal{A}: \real \left[1+\frac{zf''(z)}{f'(z)}\right]>0,\,(z\in\D)
\right],\\
\mathcal{S}^{\star}&=\left[f\in\mathcal{A}: \real \left[\frac{zf'(z)}{f(z)}\right]>0,\,(z\in\D) \right],
\end{split}
\]
denote the classes of convex  and starlike functions, respectively. It is well-known that
$$ f\in \mathcal{C} \quad \Rightarrow \quad f\in \mathcal{S}^{\star},$$
(see Duren \cite{duren}).
If the condition for convexity, $\real \left[1+\frac{zf''(z)}{f'(z)}\right]>0,$ is not satisfied for every $z\in \D$, then we cannot conclude that $f\in \mathcal{C}$, and consequently, $f\in \mathcal{S}^{\star}$.

\medskip

In this paper $f\in \mathcal{S}^{\star}$ in the terms of the expression
$$\left|\arg\left[\alpha+\frac{zf''(z)}{f'(z)}\right]\right|,\quad (z\in\D)$$
where $\alpha>1$.

\section{Maine results}
For our consideration we need the next lemma given by Nunokawa in \cite{Nuno92}.

\medskip

\noindent
{\bf Lemma A.} {\it Let $p(z)=1+p_{1}z+p_{2}z^{2}+\ldots $ be analytic in the unit disc $\D$ and $p(z)\neq0$ for $z\in\D$. Also, let suppose that there exists a point $z_{0}\in\D$ such that
 $\real \,p(z)>0$ for $|z|<|z_{0}|$ and $\real \,p(z_{0})=0$, where $p(z_{0})\neq0$, i.e.
$p(z_{0})=ia$, $a$ is real and $a\neq0$. Then we have $$\frac{z_{0}p'(z_{0})}{p(z_{0})}=ik,$$
where $k\geq\frac{1}{2}\left(a+\frac{1}{a}\right)$, when $a>0$, and
$k \leq-\frac{1}{2}\left(|a|+\frac{1}{|a|}\right)$, when $a<0$.}

\medskip

\begin{thm}\label{24-th 1}
Let $f\in\mathcal{A}$, $\alpha>1$, and let
\begin{equation}\label{eq2}
\left|\arg\left[\alpha+\frac{zf''(z)}{f'(z)}\right]\right|< {\arctg}\frac{\sqrt{3}}{\alpha-1}\quad (z\in\D).
\end{equation}
Then $f\in\mathcal{S}^{\star}$.
\end{thm}

\begin{proof}
If we put that $p(z)=\frac{zf'(z)}{f(z)}$, then $p(0)=1$ and
$\frac{zp'(z)}{p(z)}+p(z)-1 = \frac{zf''(z)}{f'(z)}$. So, we would need to prove that the following implication holds:
\begin{equation}\label{eq3}
\begin{split}
 \left|\arg\left[\frac{zp'(z)}{p(z)}+p(z)+\alpha-1 \right]\right|&<\arctg\frac{\sqrt{3}}{\alpha-1}\quad  (z\in\D)\\[2mm]
\Rightarrow \quad  \real p(z)&>0 \quad (z\in\D).
\end{split}
\end{equation}

First we will prove that $p(z)\neq0, z\in\D$. On a contrary, if there exists $z_{1}\in\D$ such that $z_{1}$ is the zero of
order $m$ of the function $p$, then
$p(z)=(z-z_{1})^{m}p_{1}(z)$, where $m$ is positive integer, $p_{1}$ is analytic in $\D$ with
$p_{1}(z_{1})\neq0$, and further
$$\frac{zp'(z)}{p(z)}=\frac{mz}{z-z_{1}}+\frac{zp_{1}'(z)}{p_{1}(z)}.$$
This means that the real part of the right hand side can tend to $-\infty$ when $z\rightarrow z_{1}$, which is a contradiction to the assumption of the theorem regarding the argument. Thus $p(z)\neq0$, $z\in\D.$

Now, let suppose that the implication \eqref{eq3} does not hold in the unit disc. It means that there exists a point
$z_{0}\in\D$ such that $\real p(z)>0$ for $|z|<|z_{0}|$ and $\real p(z_{0})=0$, where $p(z_{0})\neq0$.
If we put $p(z_{0})=ia$, $a$ is real and $a\neq0$, then by Lemma A
we have $$\frac{z_{0}p'(z_{0})}{p(z_{0})}=ik,$$
where $k\geq\frac{1}{2}\left(a+\frac{1}{a}\right)$, when $a>0$, and
$k \leq-\frac{1}{2}\left(|a|+\frac{1}{|a|}\right)$, when $a<0$.

Next, if we put $\Phi(z)=\frac{zp'(z)}{p(z)}+p(z)+\alpha-1$, then by the previous facts:
$$\real\Phi(z_{0})= \alpha-1, \quad \imag \Phi(z_{0})=k+a,$$
which for $a>0$ implies:
\[
\begin{split}
\arg\Phi(z_{0})&=\arctg\frac{k+a}{\alpha-1}\geq \arctg\frac{\frac{1}{2}\left(a+\frac{1}{a}\right)+a}{\alpha-1}\\
&=\arctg\frac{3a+\frac{1}{a}}{2(\alpha-1)}\geq \arctg\frac{\sqrt{3}}{\alpha-1},
\end{split}
\]
because the function $\varphi(a)=3a+\frac{1}{a}$ has its minimum value $\varphi(\frac{1}{\sqrt{3}})=2\sqrt{3}.$
Similarly, for $a<0$:
\[
\begin{split}
\arg\Phi(z_{0})&=\arctg\frac{k+a}{\alpha-1}\leq \arctg\frac{-\frac{1}{2}\left(|a|+\frac{1}{|a|}\right)-|a|}{\alpha-1}\\
&=-\arctg\frac{3|a|+\frac{1}{|a|}}{2(\alpha-1)}\leq -\arctg\frac{\sqrt{3}}{\alpha-1}.
\end{split}
\]
Combining the cases $a>0$ and $a<0$, we receive
$$\left|\arg \Phi(z_{0}) \right|\geq \arctg\frac{\sqrt{3}}{\alpha-1},$$
which is a contradiction to the relation \eqref{eq2}.

This show that $\real p(z)=\real\frac{zf'(z)}{f(z)}>0,$ for all $z\in\D$, i.e.
that  $f\in\mathcal{S}^{\star}$.
\end{proof}

\medskip

\begin{ex}
Let  $f\in \mathcal{A}$ is defined by the condition
\begin{equation}\label{eq4}
1+\frac{zf''(z)}{f'(z)}=(\sqrt{3}+1)\left(\frac{1+z}{1-z}\right)^{\frac{1}{2}}-\sqrt{3},
\end{equation}
where we use the principal value of the square root. For real $z$ close to -1 we have
$\real \left(1+\frac{zf''(z)}{f'(z)}\right)<0,$ which means that $f$ is not convex. On the other hand,
from \eqref{eq4} we have
$$\frac{zf''(z)}{f'(z)}+(\sqrt{3}+1)=(\sqrt{3}+1)\left(\frac{1+z}{1-z}\right)^{\frac{1}{2}}$$
and from here
\[
\begin{split}
\left|\arg \left[\frac{zf''(z)}{f'(z)}+(\sqrt{3}+1)\right]\right|
&\leq\frac{1}{2}\left|\arg\frac{1+z}{1-z}\right|\\
&<\frac{\pi}{4}=\arctg\frac{\sqrt{3}}{(\sqrt{3}+1)-1},
\end{split}
\]
which by Theorem \ref{24-th 1} (with $\alpha= \sqrt{3}+1$)  implies that $f\in\mathcal{S}^{\star}$.
\end{ex}

\medskip

Letting $\alpha$ tend to 1 in Theorem \ref{24-th 1}, we have the following well-known result.

\begin{cor}
Let $f\in\A$ and
$$\left|\arg\left[\frac{zf''(z)}{f'(z)}+1 \right]\right|<\frac{\pi}{2}\quad (z\in\D)$$
Then $f\in\mathcal{S}^{\star}.$
\end{cor}

\medskip

Further, it is easy to verify that the disc with center $\alpha$ and radius $\frac{\alpha\sqrt{3}}{\sqrt{3+(\alpha-1)^{2}}}$ is contained in the angle
$|\arg z|< \arctg\frac{\sqrt{3}}{\alpha-1}$, and so by using the result of  Theorem \ref{24-th 1}, we have that
$$\left|\frac{zf''(z)}{f'(z)}\right|=\left|\left(\frac{zf''(z)}{f'(z)}+\alpha \right)-\alpha\right|<
\frac{\alpha\sqrt{3}}{\sqrt{3+(\alpha-1)^{2}}}\quad (z\in\D) $$
implies $f\in\mathcal{S}^{\star}$.
Since the function $\varphi(\alpha)=:\frac{\alpha\sqrt{3}}{\sqrt{3+(\alpha-1)^{2}}}$, $\alpha\ge1$, attains its maximal value 2 for
$\alpha=4$, we get

\medskip

\begin{cor}
Let $f\in\A$ and
$$\left|\frac{zf''(z)}{f'(z)}\right|<2\quad (z\in\D)$$
Then $f\in\mathcal{S}^{\star}.$
\end{cor}

\medskip

In a similar way as in Theorem \ref{24-th 1}, we can consider  starlikeness problem in connection with the class defined by
$$\mathcal{G}=\left[f\in\mathcal{A}: \real \left[1+\frac{zf''(z)}{f'(z)}\right]<\frac{3}{2},\,z\in\D\right].$$
Ozaki in \cite{Oz1941} proved that if $f\in \mathcal{G}$, then $f$ is univalent in $\D$. Later, Umezava in
\cite{Um1952} showed the functions from $\mathcal{G}$ are convex in one direction. Also, it is shown in
the papers \cite{JoOb95} and \cite{RS1982} that $\mathcal{G}$ is subclass of $\mathcal{S}^{\star}.$

\medskip

If $\real \left[1+\frac{zf''(z)}{f'(z)}\right]<\frac{3}{2}$ is not satisfied for every $z\in \D$, then we can pose a question if
for some $\beta<1$, such that
\begin{equation}\label{eq5}
\real \left[\beta+\frac{zf''(z)}{f'(z)}\right]<\frac{3}{2}\quad (z\in \D),
\end{equation}
is it possible obtain sufficient condition for starlikeness, in a similar way as in Theorem \ref{24-th 1}.
The condition \eqref{eq5} is equivalent to
\[
\real \left[\frac{3-2\beta}{2}-\frac{zf''(z)}{f'(z)}\right]>0\quad (z\in\D),
\]
and we can use similar technique. This lead so

\begin{thm}\label{24-th 2}
Let $f\in\mathcal{A}$, $\beta<1$, and let
\[
\left|\arg\left[\frac{3-2\beta}{2}-\frac{zf''(z)}{f'(z)}\right]\right|<\arctg\frac{2\sqrt{3}}{5-2\beta} \quad (z\in\D).
\]
Then $f\in\mathcal{S}^{\star}$.
\end{thm}

\begin{proof}
As in the proof of Theorem \ref{24-th 1}, we put that $p(z)=\frac{zf'(z)}{f(z)}$, such that $p(0)=1$ and
$\frac{zp'(z)}{p(z)}+p(z)-1=\frac{zf''(z)}{f'(z)}$. Now, our aim is to prove that for some $\beta<1$
the following implication holds:
\begin{equation}\label{eq8}
\begin{split}
\left|\arg\left[\frac{5-2\beta}{2}-\frac{zp'(z)}{p(z)}-p(z)\right]\right| &<\arctg\frac{2\sqrt{3}}{5-2\beta} \quad (z\in\D)\\[2mm]
\Rightarrow \real p(z)&>0 \quad (z\in\D).
\end{split}
\end{equation}

First, as in Theorem \ref{24-th 1} we conclude that $p(z)\neq0, z\in\D$.

Further, if the implication \eqref{eq8} is not true, there exists a point
$z_{0}\in\D$ such that
 $\real p(z)>0$ for $|z|<|z_{0}|$ and $\real p(z_{0})=0$, where $p(z_{0})\neq0$.
If we put $p(z_{0})=ia$, $a$ is real and $a\neq0$, then by Lemma A
we have $$\frac{z_{0}p'(z_{0})}{p(z_{0})}=ik,$$
where $k\geq\frac{1}{2}\left(a+\frac{1}{a}\right)$, when $a>0$, and
$k \leq-\frac{1}{2}\left(|a|+\frac{1}{|a|}\right)$, when $a<0$.
Also, for
$$\Psi(z)=\frac{5-2\beta}{2}-\frac{zp'(z)}{p(z)}-p(z),$$
using the previous conclusions, we have
$$\real \Psi(z_{0})=\frac{5-2\beta}{2} , \quad \imag\Psi(z_{0})=-(k+a).$$

Using the same method as in the proof of Theorem \ref{24-th 1}, we easily conclude that
$$\arg\Psi(z_{0})=-\arctg\frac{2(k+a)}{5-2\beta}\leq -\arctg\frac{2\sqrt{3}}{5-2\beta},$$
when $a>0,$ and
$$\arg\Psi(z_{0})=-\arctg\frac{2(k+a)}{5-2\beta}\geq \arctg\frac{2\sqrt{3}}{5-2\beta},$$
when $a<0$. These facts imply contradiction of the assumption in \eqref{eq8}.

So, we have the statement of this theorem.
\end{proof}

\medskip

Letting $\beta$ tend to 1 in Theorem \ref{24-th 2}, we receive

\medskip

\begin{cor}
Let $f\in\A$ and
$$\left|\arg\left[\frac{1}{2}-\frac{zf''(z)}{f'(z)}\right]\right|<\arctg\frac{2}{\sqrt{3}}\approx49.1^\circ\quad (z\in\D).$$
Then $f\in\mathcal{S}^{\star}.$
\end{cor}

\medskip

As in the case of Theorem \ref{24-th 1}, we can show that the disc with center $\frac{3-2\beta}{2}$ and radius
$\frac{(3-2\beta)\sqrt{3}}{\sqrt{12+(5-2\beta)^{2}}}$ is containing in the angle
$|\arg z|\leq \arctg\frac{2\sqrt{3}}{5-2\beta} $. So, if
$$\left|\frac{zf''(z)}{f'(z)}\right|=
\left|\left(\frac{3-2\beta}{2}-\frac{zf''(z)}{f'(z)}\right)-\frac{3-2\beta}{2}\right|<
\frac{(3-2\beta)\sqrt{3}}{\sqrt{12+(5-2\beta)^{2}}},$$
for all $z\in\D$, then Theorem \ref{24-th 2} brings that $f\in\mathcal{S}^{\star}.$
Since the function $\psi(\beta)=:\frac{(3-2\beta)\sqrt{3}}{\sqrt{12+(5-2\beta)^{2}}}$ is decreasing
on the interval $(-\infty,1]$, and $ \psi(\beta)\rightarrow \sqrt{3}$, when $\beta \rightarrow -\infty,$ we get

\medskip

\begin{cor}
Let $f\in\A$ and
$$\left|\frac{zf''(z)}{f'(z)}\right|<\sqrt{3}\quad (z\in\D).$$
Then $f\in\mathcal{S}^{\star}.$
\end{cor}

\medskip

For a starlike function $f$ it is not necessary that $Re\frac{f(z)}{z}>0,$ $z\in\D$. For example,
for Koebe function $k(z)=\frac{z}{(1-z)^{2}}$ we have
$$\left.\real \frac{k(z)}{z}\right|_{z=(1+i)/\sqrt{2}} = -(\sqrt{2}+1)<0,$$
and this also holds for points in $\D$ close enought to $\frac{1+i}{\sqrt{2}}$.

\medskip

In the next theorem we give  a condition which provides  $\real \frac{f(z)}{z}>0$, $z\in\D$.

\begin{thm}\label{24-th 3}
Let $f\in \mathcal{A}$,\,$\gamma\geq0$,\, and let
\[
\left|\arg\left[\frac{zf'(z)}{f(z)}+\gamma \right]\right|<\arctg\frac{1}{1+\gamma}\quad (z\in\D).
\]
Then $\real \frac{f(z)}{z}>0$, $z\in\D.$
\end{thm}

\begin{proof}
We apply the same method as in the previous theorems. In this case, we use $p(z)=\frac{f(z)}{z}$, such that $p(0)=1$ and
$\frac{zp'(z)}{p(z)}+1=\frac{zf'(z)}{f(z)}$. So, the result we need to prove can be rewritten in the following equivalent form:
\begin{equation}\label{eq10}
\begin{split}
\left|\arg\left[\frac{zp'(z)}{p(z)}+1+\gamma \right]\right|&<\arctg\frac{1}{1+\gamma}\quad (z\in\D)\\[2mm]
\Rightarrow \qquad \real p(z)&>0\quad(z\in\D) .
\end{split}
\end{equation}

First, $p(z)\neq0$ for all $z\in\D$.

Next, if the implication \eqref{eq10} does not hold in the unit disc, then  there exists a point
$z_{0}\in\D$ such that $\real p(z)>0$ for $|z|<|z_{0}|$ and $\real p(z_{0})=0$, where $p(z_{0})\neq0$.
If we put $p(z_{0})=ia$, $a$ is real and $a\neq0$, then by Lemma A
we have $$\frac{z_{0}p'(z_{0})}{p(z_{0})}=ik,$$
where $k\geq\frac{1}{2}\left(a+\frac{1}{a}\right)$, when $a>0$, and
$k \leq-\frac{1}{2}\left(|a|+\frac{1}{|a|}\right)$, when $a<0$. Now, for $a>0$ we get:
\[
\begin{split}
\arg\left[\frac{z_0p'(z_0)}{p(z_0)}+1+\gamma \right]&=\arg(ik+1+\gamma)=\arctg\frac{k}{1+\gamma}\\
&\geq\frac{\frac{1}{2}\left(a+\frac{1}{a}\right)}{1+\gamma}\geq \arctg\frac{1}{1+\gamma},
\end{split}
\]
and for $a<0$,
$$\arg\left[\frac{z_0p'(z_0)}{p(z_0)}+1+\gamma \right]\leq -\arctg\frac{1}{1+\gamma}.$$
By combining the above conclusions, we receive that
$$\left|\arg\left(\frac{z_0p'(z_0)}{p(z_0)}+1+\gamma \right)\right|\geq \arctg\frac{1}{1+\gamma},$$
which is a contradiction to assumption in \eqref{eq10}.
\end{proof}

\medskip

For $\gamma=0$ in previous theorem we have next result.

\begin{cor}
Let $f\in\mathcal{A}$ and
$$\left|\arg\frac{zf'(z)}{f(z)}\right|<\frac{\pi}{4}\quad (z\in\D).$$
Then $\real \frac{f(z)}{z}>0$, $z\in\D.$
\end{cor}

\medskip

In Theorem \ref{24-th 3} the function $f$ from $\mathcal{A}$ need not be starlike as the next example shows.

\begin{ex}
Let  $f\in \mathcal{A}$ is defined by
\begin{equation}\label{eq4-2}
\frac{zf'(z)}{f(z)}=\sqrt{3}\left(\frac{1+z}{1-z}\right)^{\frac{1}{3}}+1-\sqrt{3},
\end{equation}
where we use the principal value for square root. For real $z$ close to -1 we have that
$\real \left[\frac{zf'(z)}{f(z)}\right]<0,$ which means that $f$ is not starlike. But
from \eqref{eq4-2} we have
$$\frac{zf'(z)}{f(z)}+(\sqrt{3}-1)=\sqrt{3}\left(\frac{1+z}{1-z}\right)^{\frac{1}{3}},$$
and from here
\[
\begin{split}
\left|\arg \left[\frac{zf'(z)}{f(z)}+(\sqrt{3}-1)\right]\right|
&\leq\frac{1}{3}\left|\arg\left(\frac{1+z}{1-z}\right)\right|\\
&<\frac{\pi}{6}=\arctg\frac{1}{(\sqrt{3}-1)+1},
\end{split}
\]
which by Theorem \ref{24-th 3} (with $\gamma= \sqrt{3}-1$)  implies that $\real\frac{f(z)}{z}>0,$ $z\in\D$.
\end{ex}

\medskip

It is possible to show that  the disc with center $1+\gamma$ and radius
$\frac{1+\gamma}{\sqrt{1+(1+\gamma)^{2}}}$ lies in the angle
$|\arg z|\leq \arctg\frac{1}{1+\gamma}$, and so, by using the result of  Theorem \ref{24-th 3} we have that
$$\left|\frac{zf'(z)}{f(z)}-1\right|=\left|\left(\frac{zf'(z)}{f(z)}+\gamma \right)-(1+\gamma)\right|<\frac{1+\gamma}{\sqrt{1+(1+\gamma)^{2}}}\quad (z\in\D)$$
implies $\real \frac{f(z)}{z}>0,$ $z\in\D.$
When $\gamma\rightarrow+\infty$ we receive

\medskip

\begin{cor}
Let $f\in\A$ and
$$\left|\frac{zf'(z)}{f(z)}-1\right|<1 \quad (z\in\D).$$
Then $\real \frac{f(z)}{z}>0,$ $z\in\D.$
\end{cor}

\end{document}